\documentclass[twoside,11pt]{article}

\usepackage{jmlr2e}

\usepackage{geometry}
\usepackage{dsfont}
\usepackage[utf8]{inputenc} 
\usepackage[T1]{fontenc}
\usepackage{textcomp} 
\usepackage[english]{babel}
\usepackage[autolanguage]{numprint} 
\usepackage{graphicx} 
\usepackage{verbatim} 
\usepackage[all]{xy}
\usepackage{dsfont}
\usepackage{lmodern}
\usepackage{latexsym}
\usepackage{mathtools}
\usepackage{relsize}
\usepackage{exscale}
\usepackage{tocloft}
\usepackage{bigints}
\usepackage[ruled]{algorithm2e}
\usepackage{undertilde}
\usepackage{bm}
\usepackage{enumitem}
\usepackage{algorithmic}
\usepackage{upgreek}
\usepackage{minitoc}

\renewcommand{\leq}{\leqslant}
\renewcommand{\geq}{\geqslant}
\renewcommand{\P}{\mathbb{P}}

\renewcommand{\phi}{\varphi}
\renewcommand{\epsilon}{\varepsilon}

\newcommand{\R}{\mathbb{R}}
\newcommand{\N}{\mathbb{N}}
\newcommand{\ind}{\mathds{1}}%{\mathbb{I}}
\newcommand{\E}{\mathbb{E}}

\newcommand{\kl}{\mathrm{kl}}
\renewcommand{\P}{\mathbb{P}}

\renewcommand{\geq}{\geqslant}
\renewcommand{\leq}{\leqslant}

\renewcommand{\phi}{\varphi}
\renewcommand{\d}{\mathrm{d}}

\newcommand{\cU}{\mathcal{U}}
\newcommand{\cS}{\mathcal{S}}
\newcommand{\F}{\mathcal{F}}

\newcommand{\cE}{\mathcal{E}}

\newcommand{\cN}{\mathcal{N}}
\newcommand{\Ng}{\mathcal{N}}
\newcommand{\M}{\mathcal{M}}
\newcommand{\I}{\mathcal{I}}
\newcommand{\alt}{\mathcal{A}\textit{lt}}

\let\ln\relax
\DeclareMathOperator{\ln}{log}

\newcommand{\hmu}{\widehat{\mu}}
\newcommand{\hlambda}{\widehat{\lambda}}
\newcommand{\lambdaStar}{\lambda^*}
\newcommand{\ha}{\widehat{a}}

\newcommand{\tomega}{\widetilde{\omega}}

\newcommand{\tlambda}{\widetilde{\lambda}}

\newcommand{\astar}{a^{\star}}

\DeclareMathOperator*{\argmax}{arg\,max}
\DeclareMathOperator*{\argmin}{arg\,min}
\DeclareMathOperator*{\loglog}{\log\!\log}

\DeclareMathOperator*{\conv}{Conv}

\ShortHeadings{Thresholding Bandit for Dose-ranging: The Impact of Monotonicity}{Garivier, M{\'e}nard and Rossi}
\firstpageno{1}

\begin{document}

\title{Thresholding Bandit for Dose-ranging:\ \\ The Impact of Monotonicity}

\author{\name Aur{\'e}lien Garivier \email aurelien.garivier@math.univ-toulouse.fr \\
       \addr Institut de Math\'ematiques de Toulouse; UMR5219\\
       Universit\'e de Toulouse; CNRS\\
	   UPS IMT, F-31062 Toulouse Cedex 9, France
       \AND
     \name Pierre M{\'e}nard \email pierre.menard@math.univ-toulouse.fr \\
       \addr Institut de Math\'ematiques de Toulouse; UMR5219\\
       Universit\'e de Toulouse; CNRS\\
	   UPS IMT, F-31062 Toulouse Cedex 9, France
	   \AND
	    \name Laurent Rossi \email laurent.rossi@math.univ-toulouse.fr \\
       \addr Institut de Math\'ematiques de Toulouse; UMR5219\\
       Universit\'e de Toulouse; CNRS\\
	   UPS IMT, F-31062 Toulouse Cedex 9, France
       }

%\editor{}

\maketitle

\begin{abstract}%   <- trailing '%' for backward compatibility of .sty file
We analyze the sample complexity of sequentially identifying the distribution whose expectation is the closest to some given threshold, with and without the assumption that the mean values of the distributions are increasing. In each case, we provide a lower bound valid for any risk $\delta$ and any $\delta$-correct algorithm; in addition, we propose an algorithm whose sample complexity is of the same order of magnitude. 
This work is motivated by phase 1 clinical trials, a practically important setting where the arm means are increasing by nature, and where no satisfactory solution is available so far.
\end{abstract}


% \section{Introduction}

% Probabilistic inference has become a core technology in AI,
% largely due to developments in graph-theoretic methods for the 
% representation and manipulation of complex probability 
% distributions~\citep{pearl:88}.  Whether in their guise as 
% directed graphs (Bayesian networks) or as undirected graphs (Markov 
% random fields), \emph{probabilistic graphical models} have a number 
% of virtues as representations of uncertainty and as inference engines.  
% Graphical models allow a separation between qualitative, structural
% aspects of uncertain knowledge and the quantitative, parametric aspects 
% of uncertainty...\\

% {\noindent \em Remainder omitted in this sample. See http://www.jmlr.org/papers/ for full paper.}

\section{Introduction}
The phase 1 of clinical trials is devoted to the testing of a drug on healthy volunteers for \emph{dose-ranging}. The first goal is to determine the maximum tolerable dose (MTD), that is the maximum amount of the drug that can be given to a person before adverse effects become intolerable or dangerous. A target tolerance level is chosen (typically $33\%$), and the trials aim at identifying quickly which is the dose entailing the toxicity coming closest to this level. % (under the tacit assumption that the medical efficiency of the drug increases with the dose). 
Classical approaches are based on dose escalation, and the most well-known is the "traditional 3+3 Design": see~\cite{LeTourneau09escalation,Genovese863} for and references therein for an introduction.

We propose in this chapter a complexity analysis for a simple model of phase 1 trials, which captures the essence of this problem.
We assume that the possible doses are $x_1<\ldots<x_K$, for some positive integer $K$. The patients are treated in sequential order, and identified by their rank. When the patient number $t$ is assigned a dose $x_k$, we observe a measure of toxicity $X_{k,t}$ which is assumed to be an independent random variable. Its distribution $\nu_k$ characterizes the toxicity level of dose $x_k$. To avoid obfuscating technicalities, we treat here the case of Gaussian laws with known variance and unknown mean, but some results can easily be extended to other one-parameter exponential families such as Bernoulli distributions. 
The goal of the experiment is to identify as soon as possible the dose $x_k$ which has the  toxicity level $\mu_k$ closest to the target admissibility level $S$, with a controlled risk $\delta$ to make an error. 

\paragraph{Content.} 
This setting is an instance of the \emph{thresholding bandit problem}: we refer to~\cite{LocatelliGC16} for an important contribution and a nice introduction in the fixed budget setting. Contrary to previous work, we focus here on identifying the \emph{exact sample complexity} of the problem: we want to understand precisely (with the correct multiplicative constant) how many samples are necessary to take a decision at risk $\delta$. We prove a lower bound which holds for all possible algorithms, and we propose an algorithm which matches this bound asymptotically when the risk $\delta$ tends to~$0$.

But the classical thresholding bandit problem does not catch a key feature of phase 1 clinical trials: the fact that the toxicity is \emph{known in hindsight} to be \emph{increasing} with the assigned dose. In other words, we investigate how many samples can be spared by algorithms using the fact that $\mu_1<\mu_2<\ldots<\mu_K$. Under this assumption, we prove another lower bound on the sample complexity, and provide an algorithm matching it. The sample complexity does not take a simple form (like a sum of inverse squares), but identifying it \emph{exactly} is essential even in practice, since it is the \emph{only way known so far} to construct an algorithm which reaches the lower bound. 

We are thus able to quantify, for each problem, how many samples can be spared when means are sorted, at the cost of a slight increase in the computation cost of the algorithm.

\paragraph{Connections to the State of the Art.}
Phase 1 clinical trials have been an intense field of research in the statistical community (see~\citet{LeTourneau09escalation} and references therein), but not considered as a sequential decision problem using the tools of the bandit literature. The important progress made in the recent years in the understanding of bandit models has made it possible to shed a new light on this issue, and to suggest very innovative solutions. The closest contribution are the works of~\cite{LocatelliGC16} and~\cite{chen14comb}, which provides a general framework for combinatorial pure exploration bandit problems. This work  tackles the more specific issue of phase 1 trials. It aims at providing strong foundations for such solutions: it does not yet tackle all the ethical and practical constraints. Observe that it might also be relevant to look for the highest dose with toxicity \emph{below} the target level: we discuss this variant in Section~\ref{sec:below}; however, it seems that practitioners do not consider this alternative goal in priority.

From a technical point the view, the approach followed here extends the theory of Best-Arm Identification initiated by~\cite{KaCaGa16} to a different setting. Building on the mathematical tools of that paper, we analyze the \emph{characteristic time} of a thresholding bandit problem with and without the assumptions that the means are increasing. Computing the complexity with such a structural constraint on the means is a challenging task that had never been done before. It induces significant difficulties in the theory, but (by using isotonic regression) we are still able to provide a simple algorithm for computing the complexity term, which is of fundamental importance in the implementation of the algorithm. The computational complexity of the resulting algorithm is discussed in Section~\ref{sec:pratical_implementation}.

\paragraph{Organization.} These lower bounds are presented in Section~\ref{sec:lower_bounds_threshold}. We compare the complexities of the non-monotonic case versus the increasing case. This comparison is particularly simple and enlightening when $K=2$, a setting often referred to as \emph{A/B testing}. We discuss this case in Section~\ref{sec:two_arms}, %, beginning with the somewhat artificial but pedagogically interesting case where the toxicity gap between the two doses in known.
which furnishes a gentle introduction to the general case. 
We present in Section~\ref{sec:asymptotically_optimal_algorithm} an algorithm and show that it is asymptotically optimal when the risk $\delta$ goes to $0$. The implementation of this algorithm requires, in the increasing case, an involved optimization which relies on constraint sub-gradient ascent and \emph{unimodal regression}: this is detailed in Section~\ref{sec:pratical_implementation}. 
Section~\ref{sec:numerical_exp} shows the results of some numerical experiments for different strategies with high level of risk that complement the theoretical results. 
Section~\ref{sec:below} discusses the interesting, but simpler variant of the problem where the goal is to identify the arm with mean closest, but also \emph{below} the threshold.
Section~\ref{sec:conclusion} summarizes further possible developments, and precedes most of the technical proofs which are given in appendix.

\subsection{Notation and Setting}
For $K\geq 2$, we consider a Gaussian bandit model $\big( \Ng(\mu_1,1),\ldots,\Ng(\mu_K,1)\big)$, which we unambiguously refer to by the vector of means $\bm \mu=\big(\mu_1,\ldots,\mu_K\big)$. Let $\P_{\bm \mu}$ and $\E_{\bm \mu}$ be respectively the probability and the expectation under the Gaussian bandit model $\bm \mu$. A threshold $S\in\R$ is given, and we denote by $a^*_{\bm \mu}\in\argmin_{1\leq a\leq K}|\mu_a-S|$ any optimal arm.

Let $\M$ be the set of Gaussian bandit models with an unique optimal arm and $\I=\{ \bm \mu\in \M:\, \mu_1<...<\mu_K\}$ be the subset of models with increasing means.

\paragraph{Definition of a $\delta$-correct algorithm.} 
A risk level $\delta\in(0,1)$ is fixed. 
At each step $t\in\N^*$ an agent chooses an arm $A_t \in \{1,\ldots,K\}$ and receives a conditionally independent reward $Y_t\sim\cN(\mu_{A_t},1)$. Let $\F_{t}=\sigma (A_1,Y_1,\ldots, A_t,Y_t)$ be the information available to the player at step $t$. Her goal is to identify the optimal arm $a^*_{\bm \mu}$ while minimizing the number of draws $\tau$. To this aim, the agent needs:
\begin{itemize}[noitemsep,nolistsep]
\item a \textbf{sampling rule} $(A_t)_{t\geq 1}$, where $A_t$ is $\F_{t-1}$-measurable,
\item a \textbf{stopping rule} $\tau_\delta$, which is a stopping time with respect to the filtration $(\F_t)_{t\geq 1}$,
\item a $\F_{\tau_\delta}$-measurable \textbf{decision rule} $\ha_{\tau_\delta}$.
\end{itemize}
For any setting $\cS\in \{\M,\I\}$ (the non-monotonic or the increasing case), an algorithm is said to be  $\delta$-correct on $\cS$ if for all $\bm\mu \in \cS$ it holds that $\P_{\bm \mu}(\tau_\delta<+\infty)=1$ and  $\P_{\bm \mu}(\ha_{\tau_\delta}\neq a_{\bm \mu}^*)\leq \delta$.

\section{Lower Bounds}
\label{sec:lower_bounds_threshold}

For $\cS\in\{\M,\I\}$, we define the set of \emph{alternative bandit problems} of the bandit problem $\bm \mu\in \M$ by 
\begin{equation}
\alt(\bm \mu,\cS):=\{ \bm \lambda \in \cS : a_{\bm \lambda}^*\neq a_{\bm \mu}^* \} \,,
\label{eq:def_alt}
\end{equation}
and the probability simplex of dimension $K-1$ by $\Sigma_K$.  
The first result of this chapter is a lower bound on the sample complexity of the thresholding bandit problem, which we show in the sequel to be tight when $\delta$ is small enough.
\begin{theorem}
Let $\cS\in \{\M,\I\}$ and $\delta \in (0,1/2]$. For all $\delta$-correct algorithm on $\cS$ and for all bandit models $\bm \mu \in \cS$,
\begin{equation}
    \E_{\bm \mu}[\tau_\delta]\geq T_{\cS}^*(\bm \mu)\,\kl(\delta,1-\delta)\,,
\label{eq:general_lb_non_asympt}
\end{equation}
where the characteristic time $T_{\cS}^*(\bm \mu)$ is given by
\begin{equation}
    T_{\cS}^*(\bm \mu)^{-1}=\sup_{\omega\in\Sigma_K}\inf_{\bm \lambda \in \alt(\bm \mu, \cS)} \sum_{a=1}^{K} \omega_a \frac{(\mu_a - \lambda_a)^2}{2}\,.
\label{eq:general_characteristic_time}
\end{equation}
In particular, this implies that
\begin{equation*}
    \liminf\limits_{\delta\rightarrow0} \frac{\E_{\bm \mu}[\tau_\delta]}{\log(1/\delta)}\geq T_{\cS}^*(\bm \mu)\,.
\label{eq:general_lb_asympt}
\end{equation*}
\label{th:general_lb}
\end{theorem}
This result is a generalization of Theorem~1 of \citet{garivier2016optimal}: the classical Best Arm Identification problem is a particular case of our non-monotonic setting $\cS=\M$ with an infinite threshold $S=+\infty$. It is proved along the same lines. As \citet{garivier2016optimal}, one proves that the supremum and the infimum are reached at a unique value, and in the sequel we denote by $\omega^*(\bm \mu)$  the optimal weights
\begin{equation}
\label{eq:def_omega_star}
\omega^*(\bm\mu):=\argmax_{\omega\in\Sigma_K} \inf_{\bm \lambda \in \alt(\bm \mu, \cS)} \sum_{a=1}^{K} \omega_a \frac{(\mu_a - \lambda_a)^2}{2}\,.
\end{equation}

\subsection{The Two-armed Bandit Case}
\label{sec:two_arms}
As a warm-up, we treat in the section the case $K=2$. Here (only), one can find an explicit formula for the characteristic times.
\begin{proposition}
When $K=2$, 
\begin{align}
T_{\I}^*(\bm \mu)^{-1}&= \frac{(2S-\mu_{1}-\mu_{2})^2}{8}, \label{eq:lb_2arm_unknown_gap_monotonic}\\
T_{\M}^*(\bm \mu)^{-1}&=\frac{ \min\!\left((2S-\mu_1-\mu_2)^2,(\mu_1-\mu_2)^2\right)}{8}\,.
\label{eq:lb_2arm_unknown_gap_non_monotonic}
\end{align}
\label{prop:lb_2arm_unknown_gap}
\end{proposition}
\begin{proof}
\label{proof:lb_2arm_unknown_gap}
The Equality~\eqref{eq:lb_2arm_unknown_gap_non_monotonic} is a simple consequence of Lemma~\ref{lem:alternative_general_input_general} proved in Section~\ref{sec:proof_complexity_non_monotonic}. It remains to treat the first Equality~\eqref{eq:lb_2arm_unknown_gap_monotonic}.
Let $\bm\mu\in\I$ and suppose, without loss of generality, that arm $2$ is optimal. Let $m=(\mu_1+\mu_2)/2$ be the mean of two arms and $\Delta=\mu_2-\mu_1$ be the gap. Noting that 
\begin{align*}
    \{\text{arm $1$ is optimal}\} &\Leftrightarrow m>S \quad \text{and} \\
    \{\text{arm $2$ optimal}\} &\Leftrightarrow m<S\,,
\end{align*}
we obtain 
\begin{align*}
    T_{\I}^*(\mu)^{-1}&=\sup_{\omega\in[0,1]}\inf_{\big\{\mu_1'<\mu_2',\, |S-\mu_1'|<|S-\mu_2'|\big\}} \frac{\omega}{2}(\mu_1-\mu_1')^2+ \frac{1-\omega}{2}(\mu_2-\mu_2')^2\\
    &= \sup_{\omega\in[0,1]}A(\omega)\,,
\end{align*}
where $m'=(\mu_1'+\mu_2')/2$, $\Delta'=\mu_2'-\mu_1'$ and we denote by $A(\omega)$ the function 
\begin{align*}
    A(\omega):=\inf_{\{\Delta'>0,\,m'>S\}}& \frac{\omega}{2}\big(m-m'-(\Delta-\Delta')/2\big)^2+ \frac{1-\omega}{2}\big(m-m'+(\Delta-\Delta')/2\big)^2\,.
\end{align*}
Writing $\chi=S-m$, easy computations lead to 
\begin{equation*}
    A(\omega)=\begin{cases}
      2\omega (1-\omega) \chi^2 \qquad\qquad\text{ if } \Delta+2(2\omega-1)\chi>0\,, \\
      \big(\chi^2+(\Delta/2)^2+(2\omega-1)\chi\Delta\big)/2 \qquad \text{else.}
    \end{cases}
\end{equation*}
Thus, since the maximum of $A$ is attained at $\omega=1/2$, we just proved that $T_{\I}^*(\mu)^{-1}=\chi^2/2$. 
\end{proof}

Note that for both alternative sets the optimal weights defined in Equation~\eqref{eq:def_omega_star} are uniform: $\omega^*=[1/2,1/2]$. If the alternative set is $\I$, the optimal alternative, i.e. the element $\bm \lambda$ of $\overline{\alt(\bm \mu,\I)}$ (the closure of $\alt(\bm \mu,\I)$) which reaches the infimum in \eqref{eq:general_characteristic_time} for the optimal weights $\omega^*$, is $\bm \lambda=[S-(\mu_2-\mu_1)/2,S+(\mu_2-\mu_1)/2]$. In words, in the optimal alternative the arms are translated in such a way that the mean of the two mean values is moved to the threshold $S$. If the alternative set is $\M$ and $\bm \mu \in \I$, the optimal alternatives can be of two different forms. If the threshold is between the two mean values, then the optimal alternative is the same as for the increasing case. Otherwise, the optimal alternative is identical to the one of Best Arm Identification (see \citet{garivier2016optimal}): $\bm \lambda=[(\mu_1+\mu_2)/2,(\mu_1+\mu_2)/2]$. Thus, if $\mu_1 \leq S \leq \mu_2$, the two characteristic times coincide, as can be seen in Figure~\ref{fig:comparison_two_arms}.
\begin{figure}%[h!]
    \centering
    \includegraphics[width=0.7\textwidth]{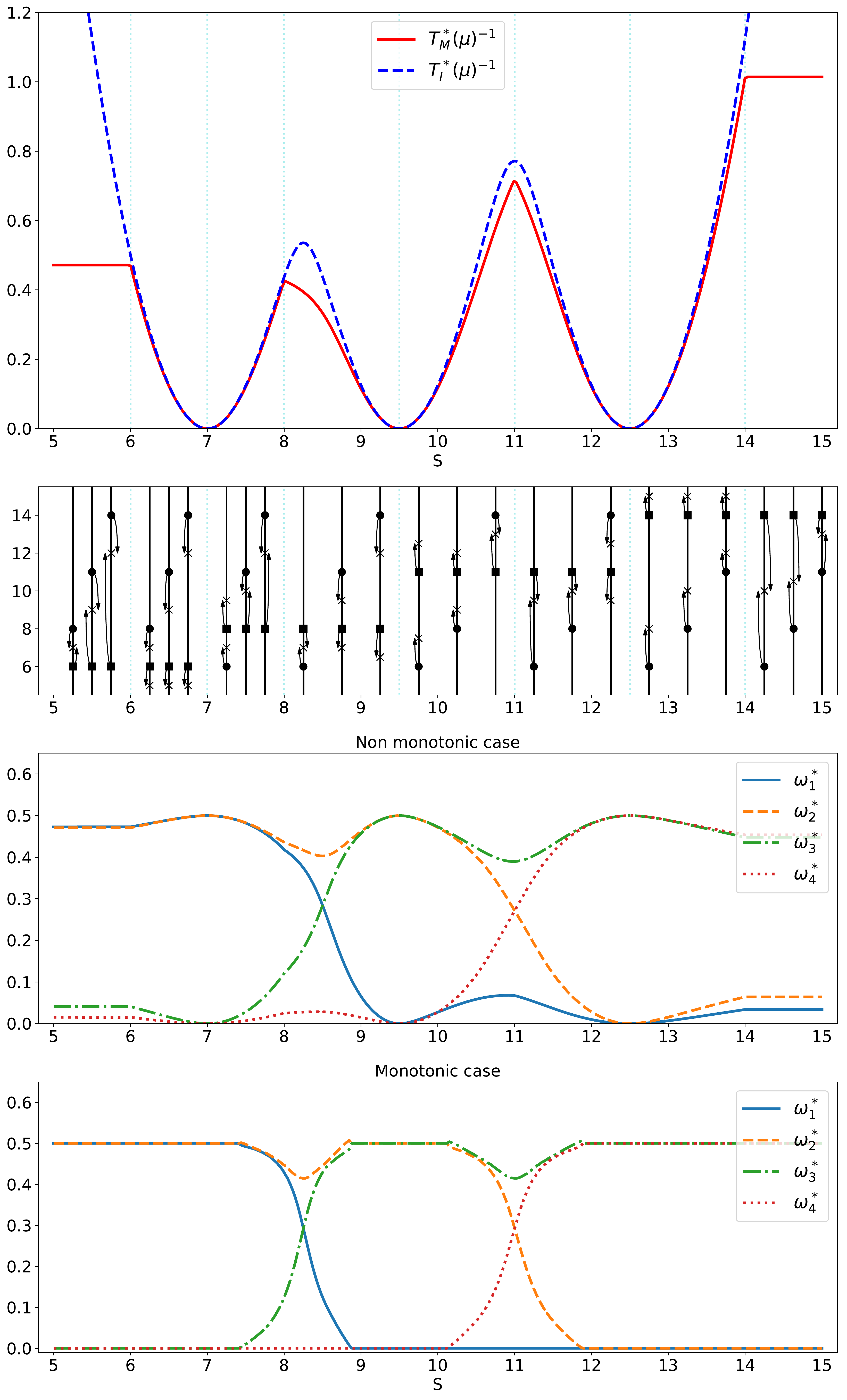}
    \caption{The complexity terms in the bandit model $\bm \mu = (6, 8, 11, 14)$.
    \textit{Top:} inverse of the characteristic time as a function of the threshold $S$; red solid line: non-monotonic case $\cS=\M$; blue dotted line: increasing case $\cS=\I$. 
    \textit{Middle:} how to move the means to get from the initial bandit model to the optimal alternative in $\M$.  
    \textit{Bottom:} the optimal weights in function of the threshold $S$.}
    \label{fig:comparison}
\end{figure}
\clearpage

\begin{figure}%[h!]
    \centering
    \includegraphics[width=0.6\textwidth]{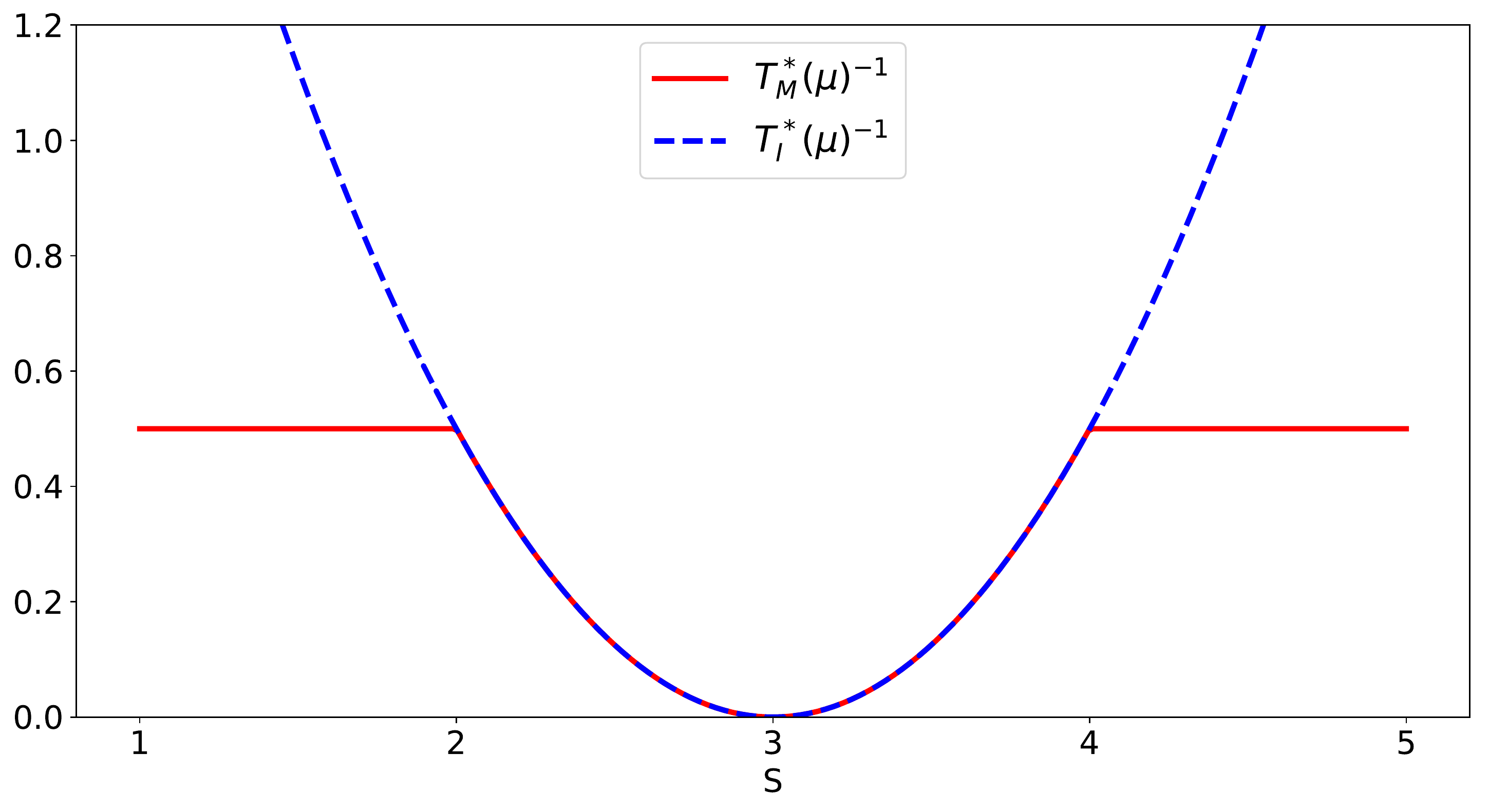}
    \caption{Inverse of the characteristic times as a function of the threshold $S$, for $\bm \mu = [2, 4]$. Solid red: general thresholding case ($\cS=\M$). Dotted blue: increasing case ($\cS=\I$).}
    \label{fig:comparison_two_arms}
\end{figure}

\subsection{On the Characteristic Time and the Optimal Proportions}
\label{sec:characteristic_time}
We now illustrate, compare and comment the different complexities for a general bandit model $\bm \mu \in \I$ with $K\geq 2$ (see Figure~\ref{fig:comparison}). Since $\I\subset\M$, it is obvious that $T_{\I}^*(\bm \mu) \leq  T_{\M}^*(\bm \mu)$. The difference $ T_{\M}^*(\bm \mu)- T_{\I}^*(\bm \mu)$ is almost everywhere positive, and can be very large. Both  $T_{\I}^*(\bm \mu)$ and  $T_{\M}^*(\bm \mu)$ tend to $+\infty$ as $S$ tends to middle of two consecutive arms. 

\paragraph{On the structure of the optimal weights in the non-monotonic case.}
\begin{lemma}
\label{lem:alternative_general_input_general}
For all $\bm \mu\in\M$,
\begin{equation}
    T_{\M}^*(\bm \mu)^{-1}= \max_{\omega\in\Sigma_K}\nonumber \min_{b\neq a^*}\frac{\omega_{a^*_{\bm \mu}} \omega_b}{2(\omega_{a^*_{\bm \mu}}+\omega_b)}\min \big((\mu_{a^*_{\bm \mu}}-\mu_b)^2,(2S-\mu_{a^*_{\bm \mu}}-\mu_b)^2\big)\,.
    \label{eq:alternative_general_input_general}
\end{equation}
\end{lemma}
In the non-monotonic case $\cS=\M$, there are two types of optimal alternatives (as in Section~\ref{sec:two_arms}). Indeed, the proof of Lemma~\ref{lem:alternative_general_input_general} in Appendix~\ref{sec:proof_complexity_non_monotonic} shows that the best alternative takes one of the two following forms. 
Either the optimal arm $\mu_{a^*_{\bm \mu}}$ and its challenger $\mu_b$ are moved to a pondered mean (by the optimal weights $\omega^*$) of the two arms (just like in the Best Arm Identification problem), leading to a constant $(\mu_{a^*_{\bm \mu}}-\mu_b)^2$ in Equation~\eqref{eq:alternative_general_input_general}. 
Or, as in the increasing case $\cS=\I$ (see the proof of Proposition~\ref{prop:lb_2arm_unknown_gap}), both arms $\mu_{a^*_{\bm \mu}}$ and $\mu_b$ are translated in the same direction, leading to the constant $(2S-\mu_{a^*_{\bm \mu}}-\mu_b)^2$. 
Figure~\ref{fig:comparison} summarizes the different possibilities on a simple example with $K=4$ arms, for different values of the threshold $S$. According to the value of $S$, the best alternative is shown in the second plot from the top.

\paragraph{On the structure of the optimal weights in the increasing case.} In the increasing case $\cS=\I$, one can show the remarkable property that \emph{the optimal weights $\omega^*(\bm\mu)$ put mass only on the optimal arm and its two closest arms}. This strongly contrasts with the non-monotonic case, as illustrated at the bottom of Figure~\ref{fig:comparison}. For simplicity we assume that $1 < a^*_{\bm\mu} < K$. Let $\tomega$ be some weights in $\Sigma_3$. Let $D^+(\theta,\tomega)$ be the cost, with weights $\tomega$, for moving from the initial bandit problem $\bm \mu$ to a bandit problem $\bm\tlambda^+$ where arm $a^*_{\bm\mu}$ has mean $\theta \leq S$ and $S$ is halfway between  $\mu_{a^*_{\bm\mu}}$ and $\mu_{a^*_{\bm\mu}+1}$,
\begin{equation*}
    \tlambda_a^+=\left\{\begin{array}{ll}
        \mu_a  & \text{ if } a>a^*_{\bm\mu}+1 \,, \\
        2S-\theta  & \text{ if } a=a^*_{\bm\mu}+1\,,\\
        \theta  & \text{ if } a=a^*_{\bm\mu}\,,\\
        \min(\theta,\mu_a)  & \text{ if } a\leq a^*_{\bm\mu}-1\,.\\ 
    \end{array}\right.
\end{equation*}
The explicit formula for $D^+(\theta,\tomega)$ is 
\begin{align*}
D^+(\theta &,\tomega)=
    \tomega_{-1}\frac{\big(\mu_{a^*_{\bm\mu}-1} - \min(\mu_{a^*_{\bm\mu}-1},\theta)\big)^2}{2}+\tomega_0\frac{(\mu_{a^*_{\bm\mu}} - \theta)^2}{2}+\tomega_{1} \frac{\big(\mu_{a^*_{\bm\mu}+1} - (2 S-\theta)\big)^2}{2}\,.
\end{align*}
Similarly we can do the same with arm $a_{\bm\mu}^*-1$: moving from $\bm\mu$ to a bandit problem $\bm\tlambda^-$, defined for $\theta\geq S$ by 
\begin{equation*}
    \tlambda_a^-=\left\{\begin{array}{ll}
        \mu_a  & \text{ if } a<a^*_{\bm\mu}-1 \,, \\
        2S-\theta  & \text{ if } a=a^*_{\bm\mu}-1\,,\\
        \theta  & \text{ if } a=a^*_{\bm\mu}\,,\\
        \max(\theta,\mu_a)  & \text{ if } a\geq a^*_{\bm\mu}+1\,,\\ 
    \end{array}\right.
\end{equation*}
where both arms $a^*_{\bm\mu}-1$ and $a^*_{\bm\mu}$ are optimal. For this alternative the cost is 
\begin{align*}
    D^-(\theta &,\tomega)=
    \tomega_{-1}\frac{\big(\mu_{a^*_{\bm\mu}-1} -(2 S-\theta) \big)^2}{2}+\tomega_0\frac{(\mu_{a^*_{\bm\mu}} - \theta)^2}{2}+\tomega_{1} \frac{\big(\mu_{a^*_{\bm\mu}+1} - \max(\mu_{a^*_{\bm\mu}+1},\theta)\big)^2}{2}\,.
\end{align*}
It appears, see the proof of Proposition~\ref{prop:simplified_general_characteristic_time} in Appendix~\ref{sec:exprCompIncr}, that these two types of alternative $\bm\tlambda^+$ and $\bm\tlambda^-$ are the optimal one. Note that they are also in $\overline{\alt(\bm \mu,\I)}$, the closure of the set of alternatives of $\bm\mu$. 
\begin{proposition}
For all $\bm \mu \in \I$,
\begin{align}
    T_{\I}^*(\bm \mu)^{-1}=\sup_{\tomega \in\Sigma_3} &\min\!\!\bigg( \min_{\{2S-\mu_{a^*_{\bm\mu}+1}\leq\theta\leq S\}} D^+(\theta,\tomega),\, \min_{\{S\leq \theta \leq 2S-\mu_{a^*_{\bm\mu}-1}\}} D^-(\theta,\tomega)\bigg)\,.
\label{eq:simplified_general_characteristic_time}
\end{align}
\label{prop:simplified_general_characteristic_time}
\end{proposition}
The intuition behind this proposition is that if we try to transform $\bm \mu$ into an alternative $\bm \lambda$ with $b>a^*_{\bm \mu}+1$ as optimal arm we have to pass by an alternative with optimal arm $a^*_{\bm \mu}+1$ since we impose to the means to be increasing. It remains to see that this intermediate alternative has always a smaller cost. The cases with $a^*_{\bm\mu}=1 \text{ or } K$ are similar considering only the the alternatives $\bm\tlambda^+$ if $a^*_{\bm\mu}=1$ and $\bm\tlambda^-$ if $a^*_{\bm\mu}=K$. We can also derive bounds on the characteristic time to see that the dependence in $K$ disappear. It is important to note that this property is really asymptotic when $\delta$ goes to zero and it is not clear at all that the dependence of the complexity in $K$ would also disappear for moderate value of $\delta$, we think it is not the case.
\begin{proposition}
\label{prop:bounds_characteristic_time}
For all $\bm \mu \in \I$ such that $1< a^*_{\bm\mu}<K$, considering the gaps: $\Delta_{-1}^2=(2S-\mu_{a^*_{\bm\mu}-1}-\mu_{a^*_{\bm\mu}})^2/8$, $\Delta_{1}^2=(2S-\mu_{a^*_{\bm\mu}+1}-\mu_{a^*_{\bm\mu}})^2/8$ and $\Delta_{0}^2=\min(\Delta_{-1}^2, \Delta_{1}^2)$,
\begin{equation}
    \label{eq:bounds_characteristic_time}
    \frac{1}{\Delta_0^2}\leq T_{\I}^*(\bm \mu) \leq \sum_{k=-1}^1\frac{1}{\Delta_k^2}\leq \frac{3}{\Delta_0^2}\,.
\end{equation}
\end{proposition}
%\newpage
\section{An Asymptotically Optimal Algorithm}
\label{sec:asymptotically_optimal_algorithm}
We present in this section an asymptotically optimal algorithm inspired by the \textit{Direct-tracking} procedure of \citet{garivier2016optimal} (which borrows the idea of tracking from GAFS-MAX algorithm of \citet{antos2008active}).
At any time $t\geq 1$ let $h(t)=(\sqrt{t}-K/2)_+$ (where $(x)_+$ stands for the positive part of $x$) and $U_t = \{ a : N_a(t) < h(t)\}$ be the set of "abnormally rarely sampled" arms. After $t$ rounds the empirical mean of arm a is 
\begin{equation*}
\hmu_a(t)=\hmu_{a,N_a(t)}=\frac{1}{N_a(t)}\sum_{s=1}^{t} Y_s\, \ind_{\{A_s=a\}}\,,
\end{equation*}
where $N_a(t) = \sum_{s=1}^t \ind_{\{ A_s = a \}}$ denotes the number of draws of arm $a$ up to and including time $t$.

\vspace{0.3cm}
% 	\begin{center} \fbox{\begin{minipage}{14cm}

				\begin{algorithm}[H]
					\label{alg:general_case}
					\smallskip
					\textbf{Sampling rule}\\
					\[A_{t+1} \in \left\{\begin{array}{ll}
                    \underset{a \in U_t}{\text{argmin}}  \ N_a(t) \ \text{if} \ U_t \neq \emptyset & (\textit{forced  exploration})\\
                    \underset{1 \leq a \leq K}{\text{argmax}} \ t \;w_a^*(\widehat{\bm\mu}(t)) -{N_a}(t) & (\textit{direct tracking})
                  \end{array}
\right.\]
					\textbf{Stopping rule}\\
					    \begin{align}
					        &\tau_\delta=\inf\Bigg\{ t\in \N^*:\,\widehat{\bm \mu}(t)\in \M\text{ and }
					        &\inf_{\bm\lambda \in \alt\big(\widehat{\bm \mu}(t),\cS\big)} \sum_{a=1}^{K} N_a(t) \frac{(\hmu_a(t) - \lambda_a)^2}{2}>\beta(t,\delta)
					        \Bigg\}\,.
					   \label{eq:stopping_rule}     
					    \end{align}
					\textbf{Decision rule}\\
					\begin{equation*}
					    \widehat{a}_{\tau_\delta}\in \argmin_{1\leq a\leq K}\big|\hmu_a(\tau_\delta)-S\big|\,.
					\end{equation*}
				\caption{Algorithme (Direct-tracking).}	
				\end{algorithm}
% 			\end{minipage}}
% 		\end{center}

\vspace{0.2cm}
\noindent When $L:= \mathrm{Card}\{ \argmin_{1\leq a\leq K}|\hmu_a(t)-S| \} > 1$, we adopt the convention that $T_{\cS}^*\big(\bm \hmu(t)\big)^{-1}=0$
and 
\begin{equation*}
    w_a^*(\widehat{\bm\mu}(t))=\left\{\begin{array}{ll}
        1/L & \text{if }a \in \argmin_{1\leq a\leq K}|\hmu_a(t)-S|\;, \\
        0  & \text{otherwise.} 
    \end{array}\right.
\end{equation*}

\begin{theorem}[Asymptotic optimality]~\\
For $\cS\in\{\I,\M\}$, for the constant 
$C= e^{K+1}\left(\frac{2}{K}\right)^{\!\! K} \big(2(3K+2)\big)^{3K} \frac{4}{\log(3)}$ 
%$C$ defined in Equation~\eqref{eq:def_C} of Section~\ref{sec:proofs}
and for $\beta(t,\delta)=\log(t C/\delta)+(3K+2)\loglog(t C/\delta)$, Algorithm~\ref{alg:general_case} is $\delta$-correct on $\cS$ and asymptotically optimal, in the sense that
\begin{equation}
\limsup\limits_{\delta\rightarrow 0} \frac{\E_{\bm \mu}[\tau_\delta]}{\log(1/\delta)}\leq  T_{\cS}^*(\bm \mu)\,.
\label{eq:general_ub}
\end{equation}
\label{th:general_ub}
\end{theorem}
The analysis of Algorithm~\ref{alg:general_case} is the same in both the increasing case $\cS=\I$ and the non-monotonic case $\cS=\M$. It is deferred to Appendix~\ref{sec:proofs_asymptotic_otpimality_direct_tracking}.
However, the practical implementations are quite specific to each case, and we detail them in the next section.
\subsection{On the Implementation of Algorithm~\ref{alg:general_case}}
\label{sec:pratical_implementation}
The implementation of Algorithm~\ref{alg:general_case} requires to compute efficiently the optimal weights $w^*(\bm \mu)$ given by Equation~\eqref{eq:def_omega_star}. For the non-monotonic case $\cS=\M$, one can follow the lines of \citet{garivier2016optimal}, Section~2.2 and replace their Lemma~3 by  Lemma~\ref{lem:alternative_general_input_general} above.

In the increasing case $\cS=\I$, however, implementing the algorithm is more involved. It is not sufficient to simply use Proposition~\ref{prop:simplified_general_characteristic_time}, since $\hmu(t)$ is not necessarily in $\I$. Let $\I_b:=\{\bm \lambda\in \I,a^*_{\bm \lambda}=b\}$ be the set of alternatives with $b$ as optimal arm. Noting that the function
\begin{align}
    F : w &\mapsto \inf_{\bm \lambda \in \alt(\bm \mu, \I)} \sum_{a=1}^{K} \omega_a \frac{(\mu_a - \lambda_a)^2}{2} \nonumber\\
    &=\min_{b\neq a^*_{\bm \mu}}\inf_{\bm \lambda \in \I_b} \sum_{a=1}^{K} \omega_a \frac{(\mu_a - \lambda_a)^2}{2}
\label{eq:def_F}
\end{align}
is concave (since it is the infimum of linear functions), one may access to its maximum by a sub-gradient ascent on the probability simplex $\Sigma_K$ (see e.g. \citet{boyd2003subgradient}). Let $\overline{\I}_b$ denote the closure of $\I_b$, and let
\begin{equation}
    \bm \lambda^b:=\argmin_{\lambda\in\overline{\I}_b} \sum_{a=1}^{K} \omega_a \frac{(\mu_a - \lambda_a)^2}{2}
\label{eq:def_lambda_b}
\end{equation}
be the argument of the second infimum in Equation~\eqref{eq:def_F}. The sub-gradient of $F$ at $\omega$ is 
\begin{equation*}
    \partial F(\omega)=\underset{b\in  B_{Opt}}{\conv} \left[\frac{(\mu_a - \lambda^b_a)^2}{2}\right]_{a\in\{1,\cdots,K\}}\,,
\end{equation*}
where $\conv$ denotes the convex hull operator and where $B_{Opt}$ is the set of arms that reach the minimum in \eqref{eq:def_F}.
Thus, performing the sub-gradient ascent simply requires to solve efficiently the minimization program \eqref{eq:def_lambda_b}. It appears that this problem boils down to \emph{unimodal regression} (a problem closely related to isotonic regression, see for example \citet{barlow1973statistical} and \citet{robertson1988orderrestric}). Indeed, we can rewrite the set
\begin{align*}
    &\{\bm \lambda \in \I : a^*_{\bm\lambda}=b\} = \{ \bm\lambda\in \M:\lambda_1<\ldots<\lambda_{b-1}<\\
    &\min(\lambda_b,2S-\lambda_b)\leq \max(\lambda_b,2S-\lambda_b)<\lambda_{b+1}<\ldots <\lambda_K\}\,.
\end{align*}
Assume that $\mu_b\leq S$ (the other case is similar). Then  $\lambda^b_b<S$, since $\lambda_b$ and $2S-\lambda_b$ play a symmetric role in the constraints. Thus, in this case, one may only consider the set 
\begin{align*}
    \{\bm \lambda\in \M: &\lambda_1<\ldots<\lambda_{b-1}<\lambda_b,\\
                     &2S-\lambda_K<\ldots<2S-\lambda_{b+1}<\lambda_b,\\
                     &\lambda_b\leq S\}\,.
\end{align*}
Let $\bm \lambda'$  be the new variables such that 
\begin{equation}
    \lambda_a'=\left\{\begin{array}{ll}
         \lambda_a  &\text{if } 1\leq a\leq b \,, \\
         2S-\lambda_a &\text{else} \,.
    \end{array}\right.
    \label{eq:change_of_variable_iso}
\end{equation}
Then $\bm \lambda^{b\prime}$ is the solution of the following minimization program
\begin{equation}
    \bm \lambda^{b\prime}=\argmin_{\substack{\lambda'_1\leq\ldots\leq\lambda_b'\\ \lambda_K'\leq \ldots\leq \lambda_b'\\ \lambda_b'\leq S}} \sum_{a=1}^{K} \omega_a \frac{(\mu_a' - \lambda_a')^2}{2}\,.
\label{eq:def_lambda_b_prime}
\end{equation}
Thanks to Lemma~\ref{lem:unimodal_regression_with_constraint} in Appendix~\ref{sec:technical}, it holds that
\[
 \lambda^{b'}_a=\min(\hlambda^{b}_a,S)\text{ for all } a\in \{1,\ldots,K\}\,,
\]
where
\begin{equation*}
    \bm \hlambda^{b}:=\argmin_{\substack{\lambda'_1\leq\ldots\leq\lambda_b'\\ \lambda_K'\leq \ldots\leq \lambda_b'}} \sum_{a=1}^{K} \omega_a \frac{(\mu_a' - \lambda_a')^2}{2}\,,
\end{equation*}
is the unimodal regression of $\bm \mu'$ with weights $\omega$ and with a mode located at $b$. It is efficiently computed via isotonic regressions (e.g. \citet{frisen1986unimodal}, \citet{geng1990algorithm}, \citet{mureika1992algorithm}) with a   computational complexity proportional to the number of arms $K$. From $\bm \hlambda^{b}$, one  can go back to $\bm\lambda^b$ by reversing Equation~\eqref{eq:change_of_variable_iso}. Since we need to compute $\bm\lambda^b$ for each $b\neq a^*_{\bm \mu}$, the overall cost of an evaluation of the sub-gradient is proportional to $K^2$.

\subsection{Numerical Experiments}
\label{sec:numerical_exp}
Table~\ref{tab:simu} presents the results of a numerical experiment of an increasing thresholding bandit. In addition to Algorithm~\ref{alg:general_case} (DT), we tried the Best Challenger (BC) algorithm with the finely tuned  stopping rule given by \eqref{eq:stopping_rule}. We also tried the Racing algorithm (R), with the elimination criterion of~\eqref{eq:stopping_rule}.
For a description of all those algorithms, see \citet{garivier2016optimal} and references therein. Finally, in order to allow comparison with the state of the art, we added the sampling rule of algorithm APT (Anytime Parameter-free Thresholding algorithm) from \citet{LocatelliGC16} in combination with the stopping rule~\eqref{eq:stopping_rule}. We chose to set the parameter $\epsilon$ of APT to be roughly equal to a tenth of the gap. It appears that the exploration function $\beta$ prescribed in Theorem~\ref{th:general_ub} is overly pessimistic. On the basis of our experiments, we recommend the use of $\beta(t,\delta)=\log\big(\big(\log(t)+1\big)/\delta\big)$ instead. It does, experimentally, satisfy the $\delta$-correctness property. 
For each algorithm, the final letter in Table~\ref{tab:simu} indicates whether the algorithm is aware ($\I$) or not ($\M$) that the means are increasing.
\begin{table}[h!]
    \centering

    \begin{tabular}{|c|c|c|c|c|c|c|}
        \hline
        %& BC-$\M$& D-tracking-$\M$ & $T_{\M}^*(\bm \mu) \log(1/\delta) $  & BC-$\I$  & D-Tracking-$\I$&
        & BC-$\M$ & R-$\M$ & DT-$\M$ & APT-$\M$& $T_{\M}^*(\bm \mu) \log\frac{1}{\delta} $\\
        \hline
        1& 3913 & 3609& 4119 & 5960 & 2033 \\
        2 & 3064& 3164 & 3098 & 3672 & 1861  \\
        \hline
        & BC-$\I$ & R-$\I$ & DT-$\I$& APT-$\I$ &$T_{\I}^*(\bm \mu)\log\frac{1}{\delta}$\\
        \hline
        1 & 483 & 494  & 611  & 1127 & 247\\
        2 & 2959 & 2906 & 3072  & 3531 & 1842\\
        \hline
    \end{tabular}    
    
    \caption{Monte-Carlo estimation (with $10000$ repetitions) of the expected number of draws $\E[\tau_\delta]$ for Algorithm~\ref{alg:general_case} and Best Challenger Algorithm in the increasing and non-monotonic cases. Two thresholding bandit problems are considered: bandit problem 1, $\bm \mu_1=[0.5,\, 1.1,\, 1.2,\, 1.3,\,  1.4,\,  5]$ with $S_1=1$, and bandit problem 2, $\bm \mu_2=[1,\,  2,\,  2.5]$ with $S_2=1.55$. The target risk is $\delta=0.1$ (it is approximately reached in the first scenario, while in the second the frequency of errors is of order $1\%$).}
    \label{tab:simu}
\end{table}
We consider two frameworks: in the first one, knowing that the means are increasing provides much information and gives a substantial edge: it permits to spare a large portion of the trials for the same level of risk. In the second, the complexities of the non-monotonic setting is very close to that of the increasing setting.
We chose a value of the risk $\delta$ which is relatively high ($10\%$), in order to illustrate that in this regime, the most important feature for efficiency is a finely tuned stopping rule. This shows that, even without an optimal sampling strategy, the stopping rule of~\eqref{eq:stopping_rule} is a key feature of an efficient procedure.
When the risk goes down to $0$, however, optimality really requires a sampling rule which respects the proportions of Equation~\eqref{eq:def_omega_star}, as shown by Theorem~\ref{th:general_ub}. The poor performances of APT can be explained by the crude adaptation of this algorithm to the fixed confidence setting. This possibly comes from the fact that it was originally designed for the fixed budget setting and it appears that these two frameworks are fundamentally different, as argued by~\citet{carpentier2016tight}.

\section{Closest mean below the threshold}\label{sec:below}

In this section we briefly discuss a variant of the previous problem: finding the arm with the closest mean \emph{below} the threshold, i.e. $a^*_{\bm \mu}\in\argmin_{1\leq a\leq K\,:\, \mu_a\leq S}|\mu_a-S|$, still under the assumption that the means are increasing. Surprisingly this new problem is simpler than the previous one and it is possible to compute exactly, in this case, the optimal weights and the characteristic time.

Let $\bm \mu\in\I$ be a bandit problem such that the optimal arm $a^*_{\bm \mu}$ for this new setting is unique (with $1<a^*_{\bm \mu}<K$ for the sake of clarity). As in Section~\ref{sec:characteristic_time}, we only need to consider alternative bandit problems such that arm $\astar_{\bm \mu}-1$ or $\astar_{\bm \mu}+1$ is optimal. But only one arm needs to be moved: the optimal alternative $\bm \tlambda^-$ (resp. $\bm \tlambda^+$) where $\astar_{\bm \mu}-1$ (resp. $\astar_{\bm \mu}+1$) is optimal are defined by
\begin{equation}
    \label{eq:opt_alt_below}
        \tlambda_a^-=\left\{\begin{array}{ll}
        S & \text{ if } a=a^*_{\bm\mu}\,,\\
        \mu_a  & \text{ else}, \\
    \end{array}\right.
    \qquad \text{and}\qquad 
        \tlambda_a^+=\left\{\begin{array}{ll}
        S  & \text{ if } a=a^*_{\bm\mu}+1\,,\\
        \mu_a  & \text{ else}. \\
\end{array}\right.
\end{equation}
With the same arguments used to prove Proposition~\ref{prop:simplified_general_characteristic_time}, one obtains  that 
\begin{align}
        \label{eq:charact_time_below}
            T_{\I}^*(\bm \mu)^{-1}&=\sup_{\tomega \in\Sigma_2} \min\!\!\left(\tomega^-\frac{(S-\mu_{a_{\bm \mu}^*})^2}{2},\tomega^+ \frac{(\mu_{a_{\bm \mu}^*+1}-S)^2}{2}\right)\nonumber\\
            &= \frac{1}{\dfrac{2}{(S-\mu_{a_{\bm \mu}^*})^2}+\dfrac{2}{(\mu_{a_{\bm \mu}^*+1}-S)^2}}            
            \,,
\end{align}
and that the associated optimal weights $\omega^*(\bm \mu)$ are defined by 
\begin{equation*}
    \label{eq:optimal_weights_below}
        \omega^*(\bm \mu)_a=\left\{\begin{array}{ll}
        \frac{2}{(S-\mu_{a_{\bm \mu}^*})^2}\left(\frac{2}{(S-\mu_{a_{\bm \mu}^*})^2}+\frac{2}{(\mu_{a_{\bm \mu}^*+1}-S)^2}\right)^{-1} & \text{ if } a=a^*_{\bm\mu}\,,\\
        \frac{2}{(\mu_{a_{\bm \mu}^*+1}-S)^2}\left(\frac{2}{(S-\mu_{a_{\bm \mu}^*})^2}+\frac{2}{(\mu_{a_{\bm \mu}^*+1}-S)^2}\right)^{-1} & \text{ if } a=a^*_{\bm\mu}+1\,,\\
        0 & \text{ else}. \\
    \end{array}\right.
\end{equation*}
In some way, this problem is closer to the classical threshold bandit problem  \citep{LocatelliGC16} with two arms since we are testing if the mean of arms $\astar_{\bm\mu}$ and $\astar_{\bm\mu}+1$ is below or above the threshold.

For an optimal strategy, Algorithm~\ref{alg:general_case} can be adapted to this new setting, as well as the Theorem~\ref{th:general_ub} and its asymptotic optimality proof. The only point that needs to be discussed is how to implement this algorithm in practice. One may follow the procedure described in Section~\ref{sec:pratical_implementation}: perform an gradient ascent on the simplex to compute the maximum of $F$ defined in~\eqref{eq:def_F}.  The main difficulty is to compute the sub-gradient of $F$ at $\omega\in\Sigma_K$ and in particular the projection given by~\eqref{eq:def_lambda_b}, that rewrites in this setting
\begin{equation}
    \label{eq:def_lambda_b_below}
        \bm \lambda^{b}=\argmin_{\substack{\lambda'_1\leq\ldots\leq\lambda_b'\leq S\\ S\leq \lambda_{b+1}'\leq \ldots\leq \lambda_K'}} \sum_{a=1}^{K} \omega_a \frac{(\mu_a - \lambda_a')^2}{2}\,.
\end{equation}
But this projection can also be easily computed using two isotonic regressions under bound restrictions, see for example \citet{hu1997maximum}.

\section{Conclusion}
\label{sec:conclusion}
We provided a tight complexity analysis of the \emph{dose-ranging} problem considered as a thresholding bandit problem with, and without, the assumption that the means of the arms are increasing. We proved that, surprisingly, the complexity terms can be computed almost as easily as in the best-arm identification case, despite the important constraints of our setting. We proposed a lower bound on the expected number of draws for any $\delta$-correct algorithm and adapted the  \textit{Direct-Tracking} algorithm to asymptotically reach this lower bound. We also compared the complexities of the non-monotonic and the increasing cases, both in theory and on an illustrative example. 
We showed in Section~\ref{sec:pratical_implementation} how to compute the optimal weights thanks to a sub-gradient ascent in the increasing case, a new and non-trivial task relying on unimodal isotonic regression. 
In order to complement the theoretical results, we presented some numerical experiments involving different strategies in a regime of high risk. 
In fact, despite the asymptotic nature of the results presented here, the procedure proposed here appears to be the most efficient in practice \emph{even when the number of trials implied is rather low} (which is often the case in clinical trials).

In the case where several arms are simultaneously closest to the threshold, the complexity of the problem is infinite. This suggests to extend the results presented here to the PAC setting, where the goal is to find \emph{any $\epsilon$-closest arm} with probability at least $1-\delta$. This extension, and extensions to the non-Gaussian case, are left for future investigation since they induce significant technical difficulties.

As a possibility of improvement, we can also mention the possible use of the unimodal regression algorithm of \citet{stout2000optimal} in order to compute directly~\eqref{eq:def_F} with a complexity of order $O\big(K)$.
We treated here mostly the case of Gaussian distributions with known variance. While the general form of the lower bound may easily be extended to other settings (including Bernoulli observations), the computation of the complexity terms is more involved and requires further investigations (in particular due to heteroscedasticity effects). The asymptotic optimality of Algorithm~\ref{alg:general_case}, however, can be extended directly. It remains important but very challenging tasks to make a tight analysis for moderate values of $\delta$, to measure precisely the sub-optimality of Racing and Best Challenger strategies, and to develop a more simple and yet asymptotically optimal algorithm.

% Acknowledgements should go at the end, before appendices and references

\acks{The authors thank Wouter Koolen for suggesting a key ingredient in the proof of Proposition~\ref{prop:simplified_general_characteristic_time}. The authors acknowledge the support of the French Agence Nationale de la Recherche (ANR), under grants ANR-13-BS01-0005 (project SPADRO) and ANR-13-CORD-0020 (project ALICIA).
 }


\newpage

\appendix
\section{Proofs for the Lower Bounds}
\label{sec:complexity}

\subsection{Expression of the Complexity in the Increasing Case}\label{sec:exprCompIncr}
Fix $\bm\mu\in\I$ and let $a^*$ be the optimal arm $a^*:=a_{\bm \mu}^*$. We recall the definitions of $D^+ (\theta,\tomega)$ and $D^-(\theta,\tomega)$ two functions defined  over $\R\times \Sigma_3 $  by
\begin{align}
\label{def_D_usual_case}
D^+(\theta,\tomega)&=
    \tomega_{-1}\frac{\big(\mu_{a^*-1} - \min(\mu_{a^*-1},\theta)\big)^2}{2}+\tomega_0\frac{(\mu_{a^*} - \theta)^2}{2}+\tomega_{1} \frac{\big(\mu_{a^*+1} - (2 S-\theta)\big)^2}{2}\\
D^-(\theta,\tomega)&=
    \tomega_{-1}\frac{\big(\mu_{a^*-1} -(2 S-\theta) \big)^2}{2}+\tomega_0\frac{(\mu_{a^*} - \theta)^2}{2}+\tomega_{1} \frac{\big(\mu_{a^*+1} - \max(\mu_{a^*+1},\theta)\big)^2}{2}\,,
\end{align}
if $1<a^*<K$. Else, if $a^*=1$ we define
\begin{align*}
D^+(\theta,\tomega)&=
    \tomega_0\frac{(\mu_{a^*} - \theta)^2}{2}+\tomega_{1} \frac{\big(\mu_{a^*+1} - (2 S-\theta)\big)^2}{2}\\
D^-(\theta,\tomega)&=+\infty\,,
\end{align*}
and if $a^*=K$ we define 
\begin{align*}
D^+(\theta,\tomega)&=+\infty\\
D^-(\theta,\tomega)&=
    \tomega_{-1}\frac{\big(\mu_{a^*-1} -(2 S-\theta) \big)^2}{2}+\tomega_0\frac{(\mu_{a^*} - \theta)^2}{2}\,.
\end{align*}

\begin{proof}[of Proposition~\ref{prop:simplified_general_characteristic_time}]
We just treat here the case $1<a^*<K$, the two other limit cases are very similar. We begin by proving that for all $\omega\in\Sigma_K$ 
\begin{equation}
\inf_{\bm \lambda \in \alt(\bm \mu, \cS)} \sum_{a=1}^{K} \omega_a \frac{(\mu_a - \lambda_a)^2}{2}
=\min_{b\in \{a^*-1,a^*+1\}}\inf_{\{\bm \lambda\in\I\ :\ a_{\bm \lambda}^*= b\}} \sum_{a=1}^{K} \omega_a \frac{(\mu_a - \lambda_a)^2}{2} \,.
\label{eq:temp_simga3}
\end{equation}
Indeed, let $\bm \lambda\in \I$ such that $a_{\bm \lambda}^*\notin \{a^*-1,a^*+1\}$. Suppose for example that $a_{\bm \lambda}^*< a^*-1$. Let $\bm \lambda^\alpha$ be the family of bandit problems defined for $\alpha\in [0,1]$ by
\begin{equation*}
\bm \lambda^\alpha= \alpha \bm \lambda+(1-\alpha) \bm \mu\,.
\end{equation*}
For all $\alpha \in [0,1]$, we have $\bm \lambda \in \I$. For $\bm \nu\in \I$ and $a\in\{0,..,K\}$, let $m_a(\bm \nu)=(\nu_a+\nu_{a+1})/2$ be the average of two consecutive means with the convention $m_0(\bm \nu)=-\infty$ and  $m_K(\bm \nu)=+\infty$. As in the case of two arms we have that $a_{\bm \nu}^*=a$ is equivalent to $m_a(\bm \nu)>S$ and $m_a(\bm \nu)<S$. Therefore we have the following inequalities 
\begin{align*}
&m_{a_{\bm \lambda}^*-1}(\bm \mu)<m_{a_{\bm \lambda}^*}(\bm \mu)\leq m_{a^*-2}(\bm \mu)<m_{a^*-1}(\bm \mu)<S<m_{a^*}(\bm \mu)\;\;\text{and}\\
&m_{a_{\bm \lambda}^*-1}(\bm \lambda)<S<m_{a_{\bm \lambda}^*}(\bm \lambda)\leq m_{a^*-2}(\bm \lambda)<m_{a^*-1}(\bm \lambda)<m_{a^*}(\bm \lambda)\,.
\end{align*}
Thus, by continuity of the applications $\alpha\mapsto m_{a}(\bm \lambda^\alpha)$ there exits $\alpha_0\in(0,1)$ such that 
\begin{equation*}
m_{a_{\bm \lambda}^*-1}(\bm \lambda^{\alpha_0})<m_{a_{\bm \lambda}^*}(\bm \lambda^{\alpha_0})\leq m_{a^*-2}(\bm \lambda^{\alpha_0})< S <m_{a^*-1}(\bm \lambda^{\alpha_0})<m_{a^*}(\bm \lambda^{\alpha_0})\,,
\end{equation*}
i.e. $a_{\bm \lambda^{\alpha_0}}^*=a^*-1$. But $\alpha\mapsto \sum_{a=1}^{K} \omega_a \frac{(\mu_a - \lambda_a^{\alpha})^2}{2}$ is an increasing function, and thus
\begin{equation*}
\sum_{a=1}^{K} \omega_a \frac{(\mu_a - \lambda_a^{\alpha_0})^2}{2}<\sum_{a=1}^{K} \omega_a \frac{(\mu_a - \lambda_a)^2}{2}\,.
\end{equation*}
This holds for all $\lambda$, therefore 
\begin{equation*}
    \inf_{\bm \lambda \in \alt(\bm \mu, \cS)} \sum_{a=1}^{K} \omega_a \frac{(\mu_a - \lambda_a)^2}{2}
\geq\min_{b\in \{a^*-1,a^*+1\}}\inf_{\{\bm \lambda\in\I\ :\ a_{\bm \lambda}^*= b\}} \sum_{a=1}^{K} \omega_a \frac{(\mu_a - \lambda_a)^2}{2} \,.
\end{equation*}
The reverse inequality follows from the inclusion 
\[\bigcup_{b\in\{a^*-1,a^*+1\}}\{\bm \lambda\in\I\ :\ a_{\bm \lambda}^*= b\} \subset \alt(\bm \mu,\I) \;.\]

Fix $\omega \in \Sigma_K$ and let $\bm \lambda\in \I$ be such that, say,  $a^*_{\bm \lambda}= a^*+1$ (the other case is similar). Then it implies $\lambda_{a^*}\leq S $ and we can suppose, without loss of generality, that  $\lambda_{a^*} \geq 2S -\mu_{a^*+1}$ since it holds
\begin{equation}
\label{eq:inf_temp_3_arms}
    \inf_{\{\bm \lambda\in\I\ :\ a_{\bm \lambda}^*= a^*+1\}} \sum_{a=1}^{K} \omega_a \frac{(\mu_a - \lambda_a)^2}{2}=\inf_{\{\bm \lambda\in\I\ :\ a_{\bm \lambda}^*= a^*+1,\ \lambda_{a^*} \geq 2S -\mu_{a^*+1}\}} \sum_{a=1}^{K} \omega_a \frac{(\mu_a - \lambda_a)^2}{2}\,.
\end{equation}
Let $\bm \tlambda$ be such that 
\begin{equation*}
    \tlambda_a=\left\{\begin{array}{ll}
        \mu_a  & \text{ if } a>a^*+1 \,, \\
        2S-\lambda_{a^*}  & \text{ if } a=a^*+1\,,\\
        \lambda_{a^*}  & \text{ if } a=a^*\,,\\
        \min(\lambda_{a^*},\mu_a)  & \text{ if } a\leq a^*-1\,.\\ 
    \end{array}\right.
\end{equation*}
By construction we have $\tlambda\in \overline{\{\lambda \in \I : a_{\bm \lambda}^* =a^*+1\}}$. As  $\lambda_{a^*+1}\leq 2S-\lambda_{a^*}$ and $\mu_{a^*+1}\geq 2S-\lambda_{a^*}$ hold, we get  
\begin{equation*}
    (\tlambda_{a^*+1}-\mu_{a^*+1})^2\leq  (\lambda_{a^*+1}-\mu_{a^*+1})^2\,.
\end{equation*}
Similarly, for $a\leq a^*-1 $ we have thanks to the fact that $\lambda_a \leq \lambda_{a^*}$ the inequality
\begin{equation*}
     (\tlambda_{a^*+1}-\mu_{a^*+1})^2\leq  (\lambda_{a^*+1}-\mu_{a^*+1})^2\,.
\end{equation*}
Therefore, combining these two inequalities and using the definition of $\tlambda$, one obtains
\begin{equation*}
 \sum_{a=1}^{K} \omega_a \frac{(\mu_a - \lambda_a)^2}{2}\geq \sum_{a=1}^{K} \omega_a \frac{(\mu_a - \tlambda_a)^2}{2}\,,
\end{equation*}
and we can rewrite the infimum in Equation~\ref{eq:inf_temp_3_arms}, indexing the alternative $\tlambda$ by $\theta$ the mean of arm $a$, as follows:
\begin{align}
    \inf_{\{\bm \lambda\in\I\ :\ a_{\bm \lambda}^*= a^*+1\}} \sum_{a=1}^{K} \omega_a \frac{(\mu_a - \lambda_a)^2}{2}&=\min_{2S-\mu_{a^*+1}\leq \theta \leq S}\sum_{a\leq a^*-1} \omega_a \frac{(\mu_a - \min\big(\theta,\mu_a)\big)^2}{2}\nonumber\\
    &+\omega_{a^*} \frac{(\mu_{a^*} - \theta)^2}{2}+\omega_{a^*+1} \frac{(\mu_{a^*+1} -2S+\theta)^2}{2}\nonumber\\
    &= \min_{2S-\mu_{a^*+1}\leq \theta \leq S} \sum_{a< a^*-1} \omega_a \frac{(\mu_a - \min\big(\theta,\mu_a)\big)^2}{2}\label{eq:sigma3_temp3_2}\\
    & + D^+(\theta,[\omega_{a^*-1},\omega_{a^*},\omega_{a^*+1}])\nonumber\,.
\end{align}
Similarly, if the optimal arm of the alternative is $a^*-1$, we get
\begin{align}
    \inf_{\{\bm \lambda\in\I\ :\ a_{\bm \lambda}^*= a^*-1\}} \sum_{a=1}^{K} \omega_a \frac{(\mu_a - \lambda_a)^2}{2}
    &= \min_{S \leq \theta \leq 2S-\mu_{a^*-1}} \sum_{a> a^*+1} \omega_a \frac{(\mu_a - \max\big(\theta,\mu_a)\big)^2}{2}\label{eq:sigma3_temp3_1}\\
    & + D^-(\theta,[\omega_{a^*-1},\omega_{a^*},\omega_{a^*+1}])\nonumber\,.
\end{align}
Then, by noting that 
\begin{align*}
    (\mu_a - \max\big(\theta,\mu_a)\big)^2 &\leq (\mu_{a^*+1} - \max\big(\theta,\mu_{a^*+1})\big)^2 &\forall a\leq a^*+1 \\
    (\mu_a - \min\big(\theta,\mu_a)\big)^2&\leq (\mu_{a^*-1} - \min\big(\theta,\mu_{a^*-1})\big)^2 &\forall a\leq a^*-1\,
\end{align*}
and by using the new weights $\tomega $ defined by 
\begin{equation*}
    \tomega_a=\left\{\begin{array}{cl}
        \sum_{b\leq a^*-1} w_b & \text{if } a=a^*-1\\
        \omega_a & \text{if } a=a^* \\
        \sum_{b\geq a^*+1}w_b &\text{if } a=a^*+1\\
        0 & \text{else}\,,
    \end{array}\right.
\end{equation*} 
we obtain thanks to Equation~\eqref{eq:temp_simga3} and to the fact that $\tomega$ depends only on $\omega$:
\begin{align}
\inf_{\bm \lambda \in \alt(\bm \mu, \cS)} \sum_{a=1}^{K} \omega_a \frac{(\mu_a - \lambda_a)^2}{2}\leq \min\left(\min_{\{2S-\mu_{a^*+1}\leq\theta\leq S\}} D^+(\theta,\tomega),\, \min_{\{S\leq \theta \leq 2S-\mu_{a^*-1}\}} D^-(\theta,\tomega)\right)\,,
\label{eq:sigma3_temp2}
\end{align}
where we identified $\tomega$ to an element of $\Sigma_3$. Taking the supremum on each side of \eqref{eq:sigma3_temp2}, one obtains:
\begin{equation*}
        T_{\I}^*(\bm \mu)^{-1}\leq\sup_{\tomega \in\Sigma_3}\min\!\!\bigg( \min_{\{2S-\mu_{a^*+1}\leq\theta\leq S\}} D^+(\theta,\tomega),\, \min_{\{S\leq \theta \leq 2S-\mu_{a^*-1}\}} D^-(\theta,\tomega)\bigg)\,.
\end{equation*}
In order to prove the reverse inequality and thus \eqref{eq:simplified_general_characteristic_time}, we just need to use \eqref{eq:sigma3_temp3_1}, \eqref{eq:sigma3_temp3_2} and restrict the weight $\omega$ to have a support included in $\{a^*-1,a^*,a^*+1\}$.
\end{proof}

\begin{proof}[of Proposition~\ref{prop:bounds_characteristic_time}]
We recall the definitions of the gaps: $\Delta_{-1}^2=(2S-\mu_{a^*-1}-\mu_{a^*})^2/8$, $\Delta_{1}^2=(2S-\mu_{a^*+1}-\mu_{a^*})^2/8$ and $\Delta_{0}^2=\min(\Delta_{-1}^2, \Delta_{1}^2)$. For the lower bound we consider the particular weights $\overline{\omega}\in \Sigma_{3}$ defined by 
\begin{equation*}
    \overline{\omega}_{i}=\frac{1/\Delta_i^2}{\sum_{k=-1}^{1}1/\Delta_k^2}\qquad\text{for }-1\leq i\leq 1\,.
\end{equation*}
Thanks to the Proposition~\ref{prop:simplified_general_characteristic_time}, we know that 
\begin{equation*}
    T_{\I}^*(\bm \mu)^{-1}\geq\min\!\!\bigg( \min_{\{2S-\mu_{a^*+1}\leq\theta\leq S\}} D^+(\theta,\overline{\omega}),\, \min_{\{S\leq \theta \leq 2S-\mu_{a^*-1}\}} D^-(\theta,\overline{\omega})\bigg)\,.
\end{equation*}
Then we can lower bound the two terms that appear in the minimum. Indeed, we have, denoting the mean $\overline{\theta}=\overline{\omega_0}\mu_{a^*}+\overline{\omega_1}(2S-\mu_{a^*+1})$,
\begin{align*}
     \min_{\{2S-\mu_{a^*+1}\leq\theta\leq S\}} D^+(\theta,\overline{\omega})&\geq \min_{\{2S-\mu_{a^*+1}\leq\theta\leq S\}} \overline{\omega_0}\frac{(\mu_{a^*} - \theta)^2}{2}+\overline{\omega_{1}} \frac{\big((2S-\mu_{a^*+1}) -\theta)\big)^2}{2}\\
     &=\overline{\omega_0}\frac{(\mu_{a^*} - \overline{\theta})^2}{2}+\overline{\omega_{1}} \frac{\big((2S-\mu_{a^*+1}) -\overline{\theta})\big)^2}{2}\\
     &\geq \min(\overline{\omega_{1}},\overline{\omega_{0}}) \Delta_{1}^2\geq\frac{1}{\sum_{k=-1}^{1}1/\Delta_k^2}\,,\\
\end{align*}
where we used the definition of the weights $\overline{\omega}$ for the last inequality and the fact that either $(\mu_{a^*} - \overline{\theta})^2/2\geq \Delta_1^2$ or $\big((2S-\mu_{a^*+1}) -\overline{\theta})\big)^2/2\geq \Delta_1^2$ since by definition $\overline{\theta}$ belongs to the interval with bounds $2S-\mu_{a^*+1}$ and $\mu_{a^*}$ for the one before. Similarly one can prove the same inequality with the second term:
\begin{equation*}
    \min_{\{S\leq \theta \leq 2S-\mu_{a^*-1}\}} D^-(\theta,\overline{\omega})\geq \frac{1}{\sum_{k=-1}^{1}1/\Delta_k^2}\,,
\end{equation*}
therefore we obtain the lower bound
\begin{equation*}
    T_{\I}^*(\bm \mu)^{-1}\geq \frac{1}{\sum_{k=-1}^{1}1/\Delta_k^2}\,.
\end{equation*}
For the upper bound we just need to choose a particular $\theta$ in order to bound one of the two terms that appears in the expression of $T_{\I}^*(\bm \mu)^{-1}$. Thus, with the choice $\theta_1=(2S-\mu_{a^*+1})/2+\mu_{a^*}/2$, we get 
\begin{align*}
    \min_{\{2S-\mu_{a^*+1}\leq\theta\leq S\}} D^+(\theta,\tomega)&\leq D^+(\theta_1,\tomega)\\
    &\leq (\tomega_{-1}+\tomega_0+\tomega_1)\frac{(2S-\mu_{a^*+1} -\mu_{a^*})^2}{8}=\Delta_1^2\,,
\end{align*}
where we used that $\big(\mu_{a^*-1} - \min(\mu_{a^*-1},\theta_1)\big)^2/2\leq (2S-\mu_{a^*+1} -\mu_{a^*})^2/8$ since $\theta_1$ is at the middle between $2S-\mu_{a^*+1}$ and $\mu_{a^*}$. In the same way with $\theta_{-1}=(2S-\mu_{a^*-1})/2+\mu_{a^*}/2$ one can obtain
\begin{equation*}
    \min_{\{2S-\mu_{a^*+1}\leq\theta\leq S\}}\leq \Delta_{-1}^2\,.
\end{equation*}
Combining these two inequalities with Proposition~\ref{prop:simplified_general_characteristic_time} leads to 
\begin{equation*}
    T_{\I}^*(\bm \mu)^{-1}\leq \Delta_{0}^2\,.
\end{equation*}
\end{proof}
\subsection{Expression of the Complexity in the Non-monotonic Case}
\label{sec:proof_complexity_non_monotonic}
\begin{proof}[of Lemma~\ref{lem:alternative_general_input_general}]
To simplify the notations we note $a^*_{\bm \mu}=a^*$. Thanks to the definition of the characteristic time, we just have to prove that
\[
    \inf_{\bm \lambda \in \alt(\bm \mu, \M)} \sum_{a=1}^{K} \omega_a \frac{(\mu_a - \lambda_a)^2}{2}=
    \min_{b\neq a^*} \frac{\omega_{a^*} \omega_b}{2(\omega_{a^*}+\omega_b)}\min \big((\mu_{a^*}-\mu_b)^2,(2S-\mu_{a^*}-\mu_b)^2\big)\,.
\]

Using that 
\[\mathrm{Alt}(\bm\mu,\M)=\bigcup_{b \neq a^*} \left\{\bm \lambda \in \M : |\lambda_b-S|<|\lambda_{a^*}-S| \right\},\]
one has 
\begin{align*}
    \inf_{\bm \lambda \in \alt(\bm \mu, \cS)} \sum_{a=1}^{K} \omega_a \frac{(\mu_a - \lambda_a)^2}{2}&=\min_{b\neq a^*}\inf_{|\lambda_b-S|<|\lambda_{a^*}-S|}\sum_{a=1}^{K} \omega_a \frac{(\mu_a - \lambda_a)^2}{2}\\
    &=\min_{b\neq a^*}\inf_{|\lambda_b-S|<|\lambda_{a^*}-S|}\omega_{a^*} \frac{(\mu_{a^*} - \lambda_{a^*})^2}{2}+\omega_b \frac{(\mu_b - \lambda_b)^2}{2}\,.
\end{align*}
Since at the infimum it holds $|\lambda_b-S|=|\lambda_{a^*}-S|$, denoting $x=\lambda_b-S$, we have $\lambda_{a^*}-S=x \text{ or } -x$. Therefore, one obtains
\begin{align*}
    \inf_{|\lambda_b-S|<|\lambda_{a^*}-S|}\omega_{a^*} \frac{(\mu_{a^*} - \lambda_{a^*})^2}{2}+\omega_b \frac{(\mu_b - \lambda_b)^2}{2}=\min 
         \!\!\Bigg(\inf_{x} \omega_a^* \frac{(\mu_{a^*} -S-x)^2}{2}+\omega_b \frac{(\mu_b - S-x)^2}{2},  \\
         \inf_{x}\omega_{a^*} \frac{(\mu_{a^*} -S+x)^2}{2}+\omega_b \frac{(\mu_b - S -x)^2}{2}\Bigg)\,.
\end{align*}
Noting that 
\begin{align*}
    \inf_{x} \omega_{a^*} \frac{(\mu_{a^*} -S-x)^2}{2}+\omega_b \frac{(\mu_b - S-x)^2}{2}&= \frac{\omega_{a^*} \omega_b}{2(\omega_{a^*}+\omega_b)}(\mu_{a^*}-\mu_b)^2\,,\\
    \inf_{x} \omega_{a^*} \frac{(\mu_{a^*} -S-x)^2}{2}+\omega_b \frac{(\mu_b - S+x)^2}{2}&=\frac{\omega_{a^*} \omega_b}{2(\omega_{a^*}+\omega_b)}(2S-\mu_{a^*}-\mu_b)^2\,,
\end{align*}
 permits to conclude. 
\end{proof}

\section{Correctness and Asymptotic Optimality of Algorithm~\ref{alg:general_case}}
\label{sec:proofs_asymptotic_otpimality_direct_tracking}

\begin{proof}[of Proposition~\ref{th:general_ub}]
We follow and slightly adapt the proof of Theorem~14 of \citet{KaCaGa16}. We fix a bandit problem $\bm \mu\in\cS$ and the constant
\begin{equation}
    C:= e^{K+1}\left(\frac{2}{K}\right)^{\!\! K} \big(2(3K+2)\big)^{3K} \frac{4}{\log(3)} \,.
    \label{eq:def_C}
\end{equation}
We begin by proving that Algorithm~\ref{alg:general_case} is $\delta$-correctness on $\cS$, then we show that it is asymptotically optimal.

\textbf{$\delta$-correctness on $\cS$}\\
We will prove in the second part of proof that $\tau$ is almost surely finite, confer~\eqref{eq:bound_E_tau}. By definition of $\tau$, the probability that the predicted arm is the wrong one is upper-bounded by
\begin{equation}
    \P_{\bm \mu}(\ha_\tau\neq  a_{\bm \mu}^*)\leq \P_{\bm \mu}\!\!\left( \exists t\in \N^*,\,  \sum_{a=1}^{K} N_a(t) \frac{(\hmu_a(t) - \mu_a)^2}{2}>\beta(t,\delta)\right)\,,
\end{equation}
where we used that $\bm \mu\in \alt\big(\widehat{\bm \mu}(t),\cS\big)$ since $\ha_\tau\neq  a_{\bm \mu}^*$\,.
Using the union bound then Theorem~\ref{th:deviation_magru} (note that $\beta(t,\delta)\geq K+1$ thanks to the choice of $C$ ) we have 
\begin{align*}
    \P_{\bm \mu}(\ha_\tau\neq  a_{\bm \mu}^*)&\leq \sum_{t=1}^{+\infty} \P_{\bm \mu}\!\!\left( \sum_{a=1}^{K} N_a(t) \frac{(\hmu_a(t) - \mu_a)^2}{2}>\beta(t,\delta)\right)\\
    &\leq \sum_{t=1}^{+\infty} e^{K+1}\left(\frac{2}{K}\right)^{\!\!K} \Big( \beta(t,\delta) \big(\ln(t)\beta(t,\delta)+1\big)\Big)^K e^{-\beta(t,\delta)}\\
    &\leq e^{K+1}\left(\frac{2}{K}\right)^{\!\! K} \sum_{t=1}^{+\infty} \frac{\big(2(3K+2)\big)^{3K}}{\log(t C/\delta)^2}\frac{\delta}{t C}\\
    &\leq e^{K+1}\left(\frac{2}{K}\right)^{\!\! K} \big(2(3K+2)\big)^{3K} \sum_{t=1}^{+\infty} \frac{1}{t\log(3 t)^2}\frac{\delta}{C}\\
     &\leq  e^{K+1}\left(\frac{2}{K}\right)^{\!\! K} \big(2(3K+2)\big)^{3K} \frac{2 }{ \log(3) } \frac{ \delta }{ C}\leq \delta\,,
\end{align*}
where in the third inequality we replaced $\beta(t,\delta)$ by its value and used in the fourth inequality, for $C\geq 3$, the following upper-bound
\begin{equation*}
    \sum_{t=1}^{+\infty} \frac{1}{t\log(3t)^2}\leq \frac{1}{\log(3)^2}+ \int_{t=1}^{+\infty} \frac{1}{t \log(3 t)^2}\d t\leq \frac{2}{\log(3)}\,.
\end{equation*}
\textbf{Asymptotic Optimality}\\
We begin by remarking that the function $\mu \rightarrow \omega^*(\mu)$ is continuous on the sets $\cS_b=\{\bm \mu \in \cS:\ a^*_{\bm \mu}=b\}$ for $b\in\{1,\ldots,K\}$. Indeed it is a consequence of Lemma~\ref{lem:alternative_general_input_general} if $\cS=\M$ and Proposition~\ref{prop:simplified_general_characteristic_time} if $\cS=\I$ and the Maximum theorem from \citet{berge1963topological}. Let $\epsilon$ be a real in $(0,1)$. From the continuity of $w^*$ in $\bm \mu$, there exists $\alpha=\alpha(\epsilon)$ such that the neighbourhood of $\bm\mu$:
\[I_\epsilon := [\mu_1 - \alpha,\mu_1 + \alpha] \times \dots \times [\mu_K - \alpha, \mu_K + \alpha]\]
is such that for all $\bm \mu' \in I_\epsilon$, 
\begin{equation*}
    \bm \mu'\in \cS, \quad a_{\bm \mu}^*=a_{\bm \mu' }^*\text{ and} \quad \max_a |w_a^*(\bm \mu') - w^*_a(\bm \mu)| \leq \epsilon\,.
\end{equation*}
Let $T \in \N^*$ and define the typical  event where $\widehat{ \bm \mu}(t)$ is not too far from $\bm\mu$
\[\cE_T(\epsilon)= \bigcap_{t = T^{1/4}}^{T}\left(\widehat{ \bm \mu}(t) \in I_\epsilon \right).\]
The two following Lemmas are extracted from \citet{KaCaGa16}.
% The following lemma is a consequence of the ``forced exploration'' performed by the algorithm, that ensures that each arm is drawn at least of order $\sqrt{t}$ times at round $t$. 

\begin{lemma}\label{lem:concSimple}There exists two constants $B,C$ (that depend on $\bm \mu$ and $\epsilon$) such that \[\P_{\bm \mu}(\cE_T^c) \leq B T \exp(-C T^{1/8})\,.\] 
\end{lemma}
\begin{lemma}\label{lem:CcqceTracking} There exists a constant $T_\epsilon$ such that for $T \geq T_\epsilon$, it holds that on $\cE_T$, 
 \[\forall t \geq \sqrt{T}, \ \max_a \left|\frac{N_a(t)}{t} - w_a^*(\mu)\right| \leq 2(K-1)\epsilon\]
\end{lemma}
We now assume that $T\geq T_\epsilon$. Introducing the constant
\begin{equation*}
C^*_\epsilon(\bm \mu) = \inf_{\substack{\bm \mu' : ||\bm\mu' - \bm \mu|| \leq \alpha(\epsilon) \\ \bm w' : ||\bm w' -{\bm w}^*(\bm \mu)||\leq 2(K-1)\epsilon}} \inf_{\lambda \in \alt\big(\bm \mu',\cS\big)} \sum_{a=1}^{K} w_a \frac{(\mu_a'(t) - \lambda_a)^2}{2}\,,
\end{equation*}
thanks to Lemma~\ref{lem:CcqceTracking}, on the event $\cE_T$ it holds that for every $t \geq \sqrt{T}$, 
\begin{equation}
t \inf_{\lambda \in \alt\big(\widehat{\bm \mu}(t),\cS\big)} \sum_{a=1}^{K} \frac{N_a(t)}{t} \frac{(\hmu_a(t) - \lambda_a)^2}{2} \geq t C^*_\epsilon(\bm \mu)\,.
\label{eq:ineq_C}
\end{equation}
Thus, combining \eqref{eq:ineq_C} and the definition of the stopping rule~\eqref{eq:stopping_rule}, we have on the event $\cE_T$
\begin{eqnarray*}
 \max(\tau_\delta,T) & \leq & \sqrt{T} + \sum_{t=\sqrt{T}}^T \ind_{\left(\tau_\delta > t\right)} \\
 & \leq & \sqrt{T} + \sum_{t=\sqrt{T}}^T \ind_{\big\{t C_\epsilon^*(\bm\mu) \leq \beta(T,\delta)\big\}} \leq \sqrt{T} + \frac{\beta(T,\delta)}{C_\epsilon^*(\bm \mu)}.
\end{eqnarray*}
Introducing 
\[T_0(\delta) = \inf \left\{ T \in \N : \sqrt{T} + \frac{\beta(T,\delta)}{C_\epsilon^*(\bm \mu)} \leq T \right\},\]
for every $T \geq \max (T_0(\delta), T_\epsilon)$, one has $\cE_T \subseteq \{\tau_\delta \leq T\}$, therefore thanks to Lemma~\ref{lem:concSimple}
\[\P_{\bm \mu}\left(\tau_\delta > T\right) \leq \P(\cE_T^c) \leq BT \exp(-C T^{1/8})\]
and
\begin{equation}
 \E_{\bm \mu}[\tau_\delta] \leq T_0(\delta) + T_\epsilon + \sum_{T=1}^\infty BT \exp(-C T^{1/8}).
 \label{eq:bound_E_tau}
\end{equation}
We now provide an upper bound on $T_0(\delta)$.  Introducing the constant 
\[H(\epsilon) = \inf \{ T \in \N : T - \sqrt{T} \geq T/(1+\epsilon)\}\]
one has 
\begin{eqnarray*}
 T_0(\delta) & \leq & H(\epsilon) + \inf \left\{T \in \N : \beta(T,\delta) \leq \frac{C_\epsilon^*(\bm\mu) T}{1+\epsilon}\right\} \\
 & \leq & H(\epsilon) + \inf \left\{T \in \N :  \log(T C/\delta)+(3K+2)\loglog(T C/\delta) \leq \frac{C_\epsilon^*(\bm\mu) T}{1+\epsilon} \right\}\,.
\end{eqnarray*}
Using technical Lemma~\ref{lem:inv_log}, for $\delta$ small enough to have $\big(C_\epsilon^*(\bm\mu)\delta\big)/\big((1+\epsilon)^2 C \big)\leq e$, we get
\begin{align*}
 T_0(\delta) &\leq C(\epsilon) + \frac{\delta}{C} \max\left( g\!\!\left(\frac{C_\epsilon^*(\bm\mu)\delta}{(1+\epsilon)^2 C}\right),\exp\!\!\left(g\!\!\left(\frac{\epsilon}{3K+2}\right) \right)\right)\\
 &\leq C(\epsilon) + \max\left( \frac{(1+\epsilon)^2 }{C_\epsilon^*(\bm\mu)}\log\!\!\left(\frac{e(1+\epsilon)^2 C}{C_\epsilon^*(\bm\mu)\delta}\log\!\!\left(\frac{ (1+\epsilon)^2 C}{C_\epsilon^*(\bm\mu)\delta}\right) \right),\frac{\delta}{C}\exp\!\!\left(g\!\!\left(\frac{\epsilon}{3K+2}\right) \right)\right)\,.
\end{align*}
This last upper bound yields, for every $\epsilon>0$, 
\[\limsup_{\delta \rightarrow 0} \frac{\E_{\bm \mu}[\tau_\delta]}{\log(1/\delta)} \leq \frac{(1+\epsilon)^2}{C_\epsilon^*(\bm \mu)}.\]
Letting $\epsilon$ tend to zero and by definition of $w^*$,
\[\lim_{\epsilon \rightarrow 0} C_\epsilon^*(\bm \mu) = T_{\cS}^*(\bm \mu)^{-1}\,,\]
allows us to conclude 
\[\limsup_{\delta \rightarrow 0} \frac{\E_{\bm \mu}[\tau_\delta]}{\log(1/\delta)} \leq T_{\cS}^*(\bm \mu)\;.\]
\end{proof}

\section{Some Technical Lemmas} \label{app:technical}\label{sec:technical}
We regroup in this Appendix some technical lemmas used in the asymptotic analysis of Algorithm~\ref{alg:general_case}.
\subsection{An Inequality}
For $0<y\leq 1/e$ let $g$ be the function 
\begin{equation}
    g(y)=\frac{1}{y}\ln\!\!\left(\frac{e}{y}\ln\!\!\left(\frac{1}{y}\right)\right)\,.
    \label{eq:def_g}
\end{equation}

\begin{lemma}
Let $A>0$ such that $1/A>e$, then for all $x\geq g(A)$
\begin{equation}
    \log(x)\leq A x\,.
\label{eq:inv_log_v0}
\end{equation}
\label{lem:inv_log_v0}
\end{lemma}
\begin{proof}
Since $g(A)\geq 1/A$, the function $x\mapsto A-1/x$ is non-decreasing, we just need to prove \eqref{eq:inv_log_v0} for $x=g(A)$. It remains to remark that 
\begin{align*}
    \log\big(g(A)\big)&\leq \log\!\!\left(\frac{2}{A}\log\!\!\left(\frac{1}{A}\right)\right)\\
    &\leq \log\!\!\left(\frac{e}{A}\log\!\!\left(\frac{1}{A}\right)\right)=A\,g(A)\,,
\end{align*}
as $\log(x)\leq x/e$.
\end{proof}

\begin{lemma}
Let  $A,\,B>0$, then for all $\epsilon\in (0,1)$ such that $(1+\epsilon)/A<e$ and $B/\epsilon>e$, for all $x\geq \max\Big(g\big(A/(1+\epsilon)\big),\exp\big(g(\epsilon/B)\big)\Big)$ 
\begin{equation}
    \log(x)+B \loglog(x) \leq A x\,.
\label{eq:inv_log}
\end{equation}
\label{lem:inv_log}
\end{lemma}
\begin{proof}
Since $\log(x)\geq g(\epsilon/B)$ thanks to Lemma~\ref{lem:inv_log_v0} we have $B\loglog(x)\leq \epsilon \log(x)$. Therefore , still using  Lemma~\ref{lem:inv_log_v0} with $x\geq g\big(A/(1+\epsilon)\big)$,
\begin{align*}
    \log(x)+B \loglog(x)&\leq (1+\epsilon)\log(x)\\
    &\leq A x\,.
\end{align*}
\end{proof}

\subsubsection{A Deviation Bound}
We recall here for self-containment the Theorem 2 of \citet{magureanu2014lipschitz}. 
\begin{theorem}
For all $\delta\geq (K+1)$ and $t\in\N^*$ we have
\begin{equation}
    \P\!\!\left(\sum_{a=1}^K N_a(t)\frac{(\hmu_a(t)-\mu_a)^2}{2}\geq \delta \right) \leq e^{K+1} \left(\frac{2\delta(\delta\log(t)+1)}{K}\right)^K e^{-\delta}\,. 
\label{eq:deviation_magru}
\end{equation}
\label{th:deviation_magru}
\end{theorem}
The factor 2 that differs from Theorem 2 of \citet{magureanu2014lipschitz} comes from the fact that we consider deviation at the right and left of the mean. 
\subsubsection{Unimodal Regression under Bound Restriction}
For $\bm \mu\in \M$, $\omega \in \mathring{\Sigma}_K$ (where $\mathring{\Sigma}_K$ stands for the interior of $\Sigma_K$) and $b\in \{1,\cdots,K\}$, let $\cU$ be the set of unimodal vector with maximum localized at $b$ 
\begin{equation}
    \label{eq:def_U}
    \cU=\{\bm\lambda : \lambda_1\leq \cdots \leq \lambda_b\geq\lambda_{b+1}\geq  \cdots\lambda_K\}\,,
\end{equation}
and $\cU_S$ be the same set with an additional bound restriction on $\lambda_b$
\begin{equation}
    \label{eq:def_U_S}
    \cU_S=\{\bm \lambda: \lambda_1\leq \cdots \leq \lambda_b\geq\lambda_{b+1}\geq  \cdots\lambda_K,\, \lambda_b\leq S\}\,.
\end{equation}
Let $\bm \hlambda$ be the unimodal regression of $\bm \mu$
\begin{equation}
    \label{eq:def_hlambda}
    \bm \hlambda:=\argmin_{\bm \lambda\in \cU} \sum_{a=1}^{K} \omega_a \frac{(\mu_a - \lambda_a)^2}{2}\,,
\end{equation}
and $\bm \lambdaStar$ be the projection of $\bm \mu $ on $\cU_S$
\begin{equation}
    \label{eq:def_lambda_star}
    \bm \lambdaStar:=\argmin_{\bm \lambda\in \cU_S} \sum_{a=1}^{K} \omega_a \frac{(\mu_a - \lambda_a)^2}{2}\,.
\end{equation}
We have, as in the case of isotonic regression (see \citet{hu1997maximum}), the following simple relation between $\bm \lambdaStar$ and $\bm \hlambda$
\begin{lemma}
\label{lem:unimodal_regression_with_constraint}
It holds that
\begin{equation*}
    \lambdaStar_a=\min(\hlambda_a,S )\text{ for all } a\in \{1,\ldots,K\}\,.
\end{equation*}
\end{lemma}
To prove Lemma~\ref{lem:unimodal_regression_with_constraint} we need the following properties on $\bm \hlambda$.
\begin{lemma}
Let $c_{-k}<\ldots<c_0>\ldots>c_l$ be real numbers and  $(A_{-k},\ldots,A_{0},\ldots,A_{k})$ be  integer intervals forming a partition of $\{1,\ldots,K\}$ be such that $\bm \hlambda$ is constant on the sets $A_i$ equals to $c_i$ for all \,
$-k\leq i \leq l$ and $b\in A_0$. Then, for all \,
$-k\leq i \leq l$ and $\bm \lambda\in\cU$
\begin{align}
\label{eq:properties_hlambda_1}
    \sum_{a\in A_i} (\mu_a -\hlambda_a)\omega_a &=0\\
    \sum_{a\in A_i} (\mu_a -\hlambda_a)\omega_a \lambda_a &\leq 0\,.
\label{eq:properties_hlambda_2}
\end{align}
\label{lem:properties_hlambda}
\end{lemma}
\begin{proof}
Since $\bm \hlambda$ is the projection of $\bm \mu$ on the closed convex $\cU$ we know that for all $\bm \lambda$ in $\cU$
\begin{equation}
    \sum_{A\in \{1,\ldots,K\}} (\mu_a-\hlambda_a)(\hlambda_a-\lambda_a)\omega_a \geq 0\,.
\label{eq:prop_projection_convex}
\end{equation}
Fix $\bm\lambda\in \cU$ and $-k\leq i\leq l$ and suppose, for example, that $i<0$. The other cases $i=0$ and $i>0$ are similar. Introduce, for $|\epsilon|< \min\big(|c_i-c_{i-1}|,|c_{i+1}-c_{i}|\big)$, the vector $\bm \lambda^\epsilon$ such that 
\[
\lambda^{\epsilon}_a=
\left\{\begin{array}{ll}
c_i-\epsilon & \text{if } a\in A_i\;,\\
\hlambda_a & \text{ else.}
\end{array}\right.
\]
By construction $\bm \lambda^\epsilon\in \cU$ and thanks to \eqref{eq:prop_projection_convex} we have 
\[
\sum_{A\in \{1,\ldots,K\}} (\mu_a-\hlambda_a)(\hlambda_a-\lambda^{\epsilon}_a)\omega_a=\epsilon\sum_{a\in A_i} (\mu_a -\hlambda_a)\omega_a \geq 0\,.
\]
Taking $\epsilon$ positive or negative proves \eqref{eq:properties_hlambda_1}. Let $x,y \in \{1,\ldots,K\}$ be such that $A_i=\{x,x+1,\ldots,y-1,y\}$ and $\bm \lambda'$ be such that 
\[
\lambda'_a=
\left\{\begin{array}{ll}
\lambda_a & \text{if } a\in A_i\;,\\
\lambda_x & \text{if } a < x\;, \\
\lambda_y & \text{if } a > y\;. \\
\end{array}\right.
\]
By construction $\bm \lambda'\in \cU$ and thanks to \eqref{eq:prop_projection_convex} we have 
\begin{align*}
 \sum_{A\in \{1,\ldots,K\}} (\mu_a-\hlambda_a)(\hlambda_a-\lambda_a)\omega_a &=\sum_{a\in A_i} (\mu_a -\hlambda_a)(c_i-\lambda'_a)\omega_a \\+ \lambda_x & \sum_{j<i}\sum_{a\in A_j} (\mu_a -\hlambda_a)\omega_a+\lambda_y\sum_{j>i}\sum_{a\in A_j} (\mu_a -\hlambda_a)\omega_a\\
 &= -\sum_{a\in A_i} (\mu_a -\hlambda_a)\lambda'_a\omega_a\,,
\end{align*}
where we used \eqref{eq:properties_hlambda_1}. Equation \eqref{eq:prop_projection_convex} allows us to prove \eqref{eq:properties_hlambda_2}.
\end{proof}

We now adapt the proof of \citet{hu1997maximum} to the case of unimodal regression.

\begin{proof}[of Lemma~\ref{lem:unimodal_regression_with_constraint}]
Since $\cU_S$ is a closed convex we just need to check that for all $\lambda \in \cU_S$
\begin{equation*}
     \sum_{a\in \{1,\ldots,K\}} \big(\mu_a -\min(\hlambda_a,S)\big)\big(\min(\hlambda_a,S)-\lambda_a\big)\omega_a\geq 0\,. 
\end{equation*}
We have, using the same notation of Lemma~\ref{lem:properties_hlambda}, 
\begin{align*}
\sum_{a\in \{1,\ldots,K\}} \big(\mu_a -\min(\hlambda_a,S)\big)\big(\min(\hlambda_a,S)-\lambda_a\big)\omega_a &=\sum_{i : c_i\leq S}\sum_{a\in A_i} (\mu_a -\hlambda_a)(\hlambda_a-\lambda_a)\omega_a\\
+\sum_{i : c_i> S}\sum_{a\in A_i} (\mu_a -S)(S-\lambda_a)&\omega_a\\
&=\sum_{i : c_i\leq S}\sum_{a\in A_i} (\mu_a -\hlambda_a)(\hlambda_a-\lambda_a)\omega_a\\
+\sum_{i : c_i> S} \sum_{a\in A_i}(\mu_a -\hlambda_a)(S-\lambda_a)\omega_a & +\sum_{i : c_i> S} \sum_{a\in A_i}(c_i
-S)(\lambda_a-S)\omega_a \geq 0\,,
\end{align*}
where we used the Lemma~\ref{lem:properties_hlambda} for the two first sums and the fact that $\lambda_a<S$ for the last sum. 
\end{proof}

\vskip 0.2in
\bibliography{biblio-BLB}

\end{document}